\documentclass[reqno]{amsart}
\usepackage[margin=1.3in]{geometry}
\usepackage{amssymb, amsfonts}
\usepackage{enumerate}
\usepackage[usenames, dvipsnames]{color}
\usepackage{verbatim}
\usepackage[hypertex]{hyperref}
\usepackage{mathrsfs}

\numberwithin{equation}{section}

\newtheorem{theorem}{Theorem}[section]

\newtheorem{lemma}[theorem]{Lemma}

\theoremstyle{definition}
\newtheorem{remark}[theorem]{Remark}

\theoremstyle{definition}

\theoremstyle{definition}

\makeatletter
\def\dashint{\operatorname%
{\,\,\text{\bf--}\kern-.98em\DOTSI\intop\ilimits@\!\!}}
\makeatother

\newcommand{\abs}[1]{\lvert#1\rvert}
\newcommand{\norm}[1]{\lVert#1\rVert}

\def\R{\mathbb{R}}      
\def\W{\Omega}          
\def\esssup{\text{ess sup}}  
\def\e{\varepsilon}             

\begin{document}

\title[The 3D Navier-Stokes equations]{Boundary $\varepsilon$-regularity criteria for the 3D Navier-Stokes equations}

\
\author[H. Dong]{Hongjie Dong}
\address[H. Dong]{Division of Applied Mathematics, Brown University, 182 George Street, Providence, RI 02912, USA}

\email{Hongjie\_Dong@brown.edu}

\thanks{H. Dong and K. Wang were partially supported by the NSF under agreement DMS-1600593.}

\author[K. Wang]{Kunrui Wang}
\address[K. Wang]{Division of Applied Mathematics, Brown University, 182 George Street, Providence, RI 02912, USA}

\email{Kunrui\_Wang@brown.edu}



\date{\today}



\begin{abstract}
We establish several boundary $\varepsilon$-regularity criteria for suitable weak solutions for the 3D incompressible Navier-Stokes equations in a half cylinder with the Dirichlet boundary condition on the flat boundary. Our proofs are based on delicate iteration arguments and interpolation techniques. These results extend and provide alternative proofs for the earlier interior results by Vasseur \cite{MR2374209}, Choi-Vasseur \cite{MR3258360}, and Phuc-Guevara \cite{Phuc}.
\end{abstract}
\maketitle

\section{Introduction and main results}
In this paper we discuss the 3-dimensional incompressible Navier-Stokes equations with unit viscosity and zero external force:
\begin{equation}\label{NS1}
\partial_tu+u\cdot\nabla u- \Delta u +\nabla p=0, \quad \text{div } u=0,
\end{equation}
where $u = (u^1(t,x),u^2(t,x),u^3(t,x))$ is the velocity field and $p = p(t,x)$ is the pressure. We consider local problem: $(t,x)\in Q$ or $Q^+$, where $Q$ and $Q^+$ denote the unit cylinder and unit half-cylinder respectively.
For  the half cylinder case, we assume that $u$ satisfies the zero Dirichlet boundary condition:
\begin{equation}\label{NS3}
u = 0  \quad\text{on}\   \{x_d=0\}\cap\partial Q^+.
\end{equation}

We are concerned with different types of $\varepsilon$-regularity criteria for suitable weak solutions for 3D Navier-Stokes equations. The suitable weak solutions are a class of Leray-Hopf weak solutions satisfying the so-called local energy inequality, which was originated by Scheffer in a series of papers \cite{Refer23,Refer24,Refer25}. The formal definition of the suitable weak solutions was first introduced by Caffarelli, Kohn, and Nirenberg \cite{Caf1}. See Section \ref{sws}.

In \cite{MR2374209} Vasseur proved the following interior $\varepsilon$-regularity criterion, which provided an alternative proof of the well-known partial regularity result for the 3D incompressible Navier-Stokes equations proved by Caffarelli, Kohn, and Nirenberg \cite{Caf1}. His proof is based on the De Giorgi iteration argument originally for elliptic equations in divergence form.
\begin{theorem}[Vasseur \cite{MR2374209}]
	\label{thm2}
For any $\tilde q\in (1,\infty)$, there exists an $\e_0=\e_0(\tilde q)>0$ such that if  $(u,p)$ is a pair of suitable weak solution to \eqref{NS1}-\eqref{NS3} in $Q$ and satisfies
	\begin{align*}
	\sup_{t\in [-1,0]}\int_{B}\abs{u(t,x)}^2\,dx + \int_Q\abs{\nabla u}^2\, dx dt+\int_{-1}^{0}\norm{p}^{\tilde q}_{L_1^{x}(B)}\, dt\leq \e_0,
	\end{align*}
	then $u$ is regular in $\overline{Q(1/2)}$.
\end{theorem}

Later Choi and Vasseur extended this result up to $\tilde{q}=1$ in \cite[Proposition 2.1]{MR3258360} with an additional condition on the maximal function of $\nabla u$.
In \cite{Phuc}, Phuc and Guevara further refined this result to the case with simply $\tilde q =1$. Their proof exploits fractional Sobolev spaces of negative order and an inductive argument in \cite{Caf1} and \cite{MR3233577}.

In this paper, we show a boundary version of Theorem \ref{thm2}. Namely,

\begin{theorem}
	\label{thm0}
For any $\tilde q\in (1,\infty)$,
there exists a universal constant $\e_0 = \e_0(\tilde q)>0$  such that if  $(u,p)$ is a pair of suitable weak solution to \eqref{NS1}-\eqref{NS3} in $Q^+$ with $p\in L_{\tilde q}^tL^x_1(Q^+)$ and satisfies
	\begin{equation*}
	\sup_{t\in [-1,0]}\int_{B^+}\abs{u(t,x)}^2\,dx + \int_{Q^+}\abs{\nabla u}^2\,dx dt+\int_{-1}^{0}\norm{p}^{\tilde q}_{L_{1}^x(B^+)}\, dt\leq \e_0,
	\end{equation*}
	then $u$ is regular in $\overline{Q^+(1/2)}$.
\end{theorem}

The condition $\tilde q>1$ is required when we apply the coercive estimate for the linear Stokes system to estimate the pressure term.
At the time of this writing, it is not clear to us whether it is possible to take $\tilde q=1$ as in the interior case.

Theorem \ref{thm0} can be used to give a new proof of the boundary partial regularity result by Seregin \cite{Refer23b}.
Another consequence of the theorem is the following boundary regularity criterion, which does not involve $\nabla u$.

\begin{theorem}
	\label{thm4}
For any $q>5/2$ and $\tilde q>1$, there exists a universal constant $\e_0 = \e_0(q,\tilde{q})$  such that if  $(u,p)$ is a pair of suitable weak solution to \eqref{NS1}-\eqref{NS3} in $Q^+$ with $p\in L_{\tilde q}^tL^x_1(Q^+)$ and satisfies
	\begin{equation*}
\|u\|_{L_q(Q^+)}+\|p\|_{L^t_{\tilde q}L^x_1(Q^+)}<\varepsilon_0,
	\end{equation*}
	then $u$ is regular in $\overline{Q^+(1/2)}$.
\end{theorem}

The above theorem is a special case of Theorem \ref{thm1b}, which will be proved in Section \ref{sec4}. The corresponding interior result when $q>20/7$ was proved recently in \cite{Phuc} by viewing the ``head pressure'' $|u|^2/2+p$ as a signed distribution, which belongs to certain fractional Sobolev spaces of negative order. This answered a question raised by Kukavica in \cite{MR2483004} about whether one can lower the exponent $3$ in the original $\varepsilon$-regularity criterion in \cite{Caf1}.
See also more recent \cite{170901382H} for an extension to the case when $q>5/2$ with an application to the estimate of box dimensions of
singular sets for the Navier--Stokes equations.
We refer the reader to \cite{Refer11, MR2559050} and references therein for various interior and boundary $\varepsilon$-regularity criteria for the Navier-Stokes equations.

The proofs of  Theorems \ref{thm0} and \ref{thm4} both rely on iteration arguments. Compared to the argument in \cite{MR2374209}, our proof of Theorem \ref{thm0} is much shorter and, in the conceptual level, closer to the original argument in \cite{Caf1}. Instead of fractional Sobolev spaces used in \cite{Phuc}, which does not seem to work for the boundary case, we consider scale invariant quantities in the usual mixed-norm Lebesgue spaces, and apply a decomposition of the pressure due to Seregin \cite{Refer23b}. We adopt some ideas in \cite{MR3129108, Dong2} on showing uniform decay rates of scale invariant quantities by induction. More precisely, we use different induction step lengths when iterating between different scale invariant quantities associated with the energy norm and the pressure respectively. In the last step, we use parabolic regularity to further improve the estimate of mean oscillation of $u$ and conclude the H\"{o}lder continuity according to Campanato's characterization of H\"{o}lder continuous functions. By a minor modification on the proof of Theorem \ref{thm0} to transform to the interior case, we also get a different proof of Theorem \ref{thm2} with refined condition $\tilde q=1$ obtained in \cite{Phuc}.
The proof of Theorem \ref{thm4}  uses a delicate interpolation argument. We treat each term on the right hand side of the generalized energy inequality in a consistent way such that they are all interpolated by the energy norms and the mixed-norm  which is assumed to be small in the condition. By fitting the exponents of those energy norms slightly less than $2$, we spare some space  that we can borrow to use Young's inequality and proceed with an iteration to obtain the desired results.

The remaining part of the paper is organized as follows. In the next section, we introduce some notation and the definition of suitable weak solutions to the Navier-Stokes equations. The proof of Theorem \ref{thm0} is given in Section \ref{sec_iter}. Section \ref{sec4} is devoted to the proof of Theorem \ref{thm1b}. In  Appendix \ref{append}, we show how to adapt our proof to the interior case where we can take $\tilde q=1$ due to a conciser estimate of the pressure.

Throughout the paper, various constants are denoted by $N$ in general, which may vary from line to line. The expression $N = N(\cdots)$ means that the given constant $N$ depends only on the contents in the parentheses.

\section{Preliminaries}         \label{sec2}
\subsection{Notation}
Let $T>0$, $\W$ be a domain in $\R^3$, $\Gamma\in \partial\Omega$, and $\W_T:=(0,T)\times\W$ with the parabolic boundary
$$
\partial_p \W_T  = [0,T) \times \partial\Omega \cup \{t=0\}\times \Omega.
$$
We denote $\dot{C}_0^{\infty}(\W,\Gamma)$ to be the space of divergence-free infinitely differentiable vector fields which vanishes near $\Gamma$. Let $\dot{J}(\W,\Gamma)$ and $\dot{J}_2^1(\W,\Gamma)$ be the closures of $\dot{C}^{\infty}_0(\W,\Gamma)$ in the spaces $L_2(\W)$ and $W^1_2(\W)$, respectively.

We shall use the following notation for balls, half balls, 
parabolic cylinders, and half parabolic cylinders:
\begin{align*}
	&B(\hat{x},\rho) = \{x\in \mathbb{R}^3 \mid \abs{x-\hat{x}}<\rho\}, \quad  B(\rho) = B(0,\rho),\quad B=B(1);\\
	&B^+(\hat{x},\rho)=\{x\in B(\hat{x},\rho)\mid x=(x_1,x_2,x_3),\ x_3>\hat{x}_3\},\\
	& B^+(\rho) = B^+(0,\rho), \quad B^+=B^+(1);\\
	&Q(\hat{z},\rho)=(\hat{t}-\rho^2,\hat{t}) \times B(\hat{x},\rho),\quad   Q(\rho)=Q(0,\rho), \quad Q=Q(1);\\
	&Q^+(\hat{z},\rho)=(\hat{t}-\rho^2,\hat{t}) \times B^+(\hat{x},\rho),\quad  Q^+(\rho)=Q^+(0,\rho), \quad Q^+=Q^+(1).
\end{align*}
where $\hat{z}=(\hat{t},\hat{x})$.

We denote
$$
A(\rho) = \esssup_{-\rho^2\le t\le 0}\frac 1 \rho\int_{B^+(\rho)}\vert u\vert ^2  \,dx,\quad
E(\rho)=\frac 1 \rho\int_{Q^+(\rho)}\vert \nabla u\vert ^2 \,dz,
$$
$$
C_{q,r}(\rho) = \rho^{1-2/q-3/r}\norm{u}_{L_q^tL_r^x(Q^+(\rho))}, \quad D_{q,r}(\rho) = \rho^{2-2/q-3/r}\norm{p-[p]_\rho(t)}_{L_q^tL_r^x(Q^+(\rho))},$$
where $q,r\in [1,\infty]$ and $[p]_\rho(t)$ is the average of $p$ with respect to $x$ in $B^+(\rho)$.
Note that all of them are scale invariant with respect to the natural scaling for \eqref{NS1}:
\begin{equation*}
u(t,x)\rightarrow \lambda
u(\lambda^2t,\lambda x),\quad p(t,x)\rightarrow \lambda^2 p(\lambda^2 t, \lambda x).
\end{equation*}


\subsection{Suitable weak solutions}
\label{sws}
The definition of suitable weak solutions was introduced in \cite{Caf1}. We say a pair $(u,p)$ is a suitable weak solution of the Navier-Stokes equations on the set $\W_T$ vanishing on $(0,T)\times \Gamma$ if

\textsl{i)} $u\in L_{\infty}(0,T;\dot{J}(\W,\Gamma))\cap L_2(0,T;\dot{J}^1_2(\W,\Gamma))$ and $p\in L_{\tilde q}^tL^x_1(\W_T)$ for some $\tilde q\ge 1$;

\textsl{ii)} $u$ and $p$ satisfy equation (\ref{NS1}) in the sense of distribution.

\textsl{iii)} For any $t\in(0,T)$ and nonnegative function $\psi\in C^{\infty}(\overline{\W_T})$ vanishing in a neighborhood of the  boundary $\{t=0\}\times \Omega$ and $(0,T)\times (\partial\Omega\setminus \Gamma)$, the integrals in the following local energy inequality are summable and the inequality holds true:
\begin{align}\label{eqn_sw_energy}
	\begin{split}
		\text{ess sup}_{0\leq s\leq t}&\int_{\Omega}\abs{u(s,x)}^2\psi(s,x)\,dx+2\int_{\Omega_t}\abs{\nabla u}^2\psi \,dx \,ds\\
		&\leq \int_{\Omega_t}\{\abs{u}^2(\psi_t+\Delta\psi)+(\abs{u}^2+2p)u\cdot\nabla\psi\}\,dx \,ds.
	\end{split}
\end{align}
We will specify the constant $\tilde q$ later so that the integrals on the right-hand side of \eqref{eqn_sw_energy} are summable.

\section{Proof of Theorem \ref{thm0}}
\label{sec_iter}

This section is devoted to the proof of Theorem \ref{thm0}. We use the abbreviation
$$
D(\rho) = D_{\tilde{q},\tilde{r}}(\rho) = \frac{1}{\rho} \norm{p-[p]_{\rho}(t)}_{L_{\tilde{q}}^tL_{\tilde{r}}^x(Q^+(\rho))},
$$
where $\tilde{q}\in (1, 2)$, $\tilde{r}\in (3/2,3)$, ${2}/{\tilde{q}}+{3}/{\tilde{r}}=3$.
We first prove a few lemmas which will be used below.
\begin{lemma}
	\label{lem1b}
For any $\rho>0$ and any pair of exponents $(q,r)$ such that
$$\frac{2}{q}+\frac{3}{r}=\frac 3 2, \quad 2\leq q\leq \infty, \quad 2\leq r\leq 6,$$
we have
$$
\rho^{-1/2}
\norm{u}_{L_{q}^tL_{r}^x(Q^+(\rho))} \leq N\left(A(\rho)+E(\rho)\right)^{\frac{1}{2}},
$$
where $N>0$ is a universal constant.
\end{lemma}
\begin{proof}
Use the standard interpolation by the Sobolev embedding inequality and H\"{o}lder's inequality.
\end{proof}

\begin{lemma}
	\label{lem4}
Let $(u,p)$ be a pair of suitable weak solution of (\ref{NS1}). For constants $\gamma \in (0,1/2]$, $\rho >0$, we have
	\begin{equation*}
		A(\gamma\rho)+E(\gamma\rho)
		\leq N\left[\gamma^{2}A(\rho)+\gamma^{-2}\left((A(\rho)+E(\rho))^{\frac{3}{2}}
+(A(\rho)+E(\rho))^{\frac{1}{2}}D(\rho)\right)\right],
	\end{equation*}
where $N>0$ is a universal constant.
\end{lemma}
\begin{proof}
The proof is more or less standard. We give the details for the sake of completeness. By scaling, we may assume $\rho=1$.
	Define the backward heat kernel as	$$
\Gamma(t,x)=\frac{1}{(4\pi(\gamma^2-t))^{3/2}}e^{-\frac{\abs{x}^2}{4(\gamma^2-t)}}.
$$
	In the energy inequality (\ref{eqn_sw_energy}), we choose $\psi=\Gamma\phi$, where $\phi\in C^{\infty}_0( (-1,1)\times B(1))$ is a suitable smooth cut-off function satisfying
	$$ 0\leq \phi \leq 1 \quad \text{in } \R \times\mathbb{R}^3, \quad \phi \equiv 1 \quad \text{ in  } Q(1/2),$$
	$$|\nabla \phi |\leq  N,\quad |\nabla^2\phi| \leq N, \quad|\partial_t\phi |\leq N \text{ in  }  \R \times\mathbb{R}^3.$$
	By using the equation
	$$\Delta \Gamma +\Gamma_t=0,$$
	we have	
\begin{align}\label{eqn8c}
	\begin{split}
		&\text{ess sup}_{-1\leq s\leq 0}\int_{B^+}\abs{u(s,x)}^2\Gamma(s,x)\phi(s,x)\,dx+2\int_{Q^+}\abs{\nabla u}^2\Gamma\phi \,dz\\
		&\leq \int_{Q^+}\{\abs{u}^2(\Gamma\phi_t+\Gamma\Delta\phi+2\nabla\phi\nabla\Gamma)+(\abs{u}^2+2\abs{p-[p] _{1}})u\cdot(\Gamma\nabla\phi+\phi\nabla\Gamma \} \,dz.
	\end{split}
\end{align}
	The test function has the following properties:
	\begin{enumerate}
		\item[(i)] For some constant $c>0$, on $Q^+(\gamma)$ it holds that
		$$\Gamma\phi = \Gamma \geq c\gamma^{-3}.$$
		\item[(ii)]For any $z\in Q^+(1)$, we have
		$$\abs{\Gamma(z)\phi(z)}\leq N\gamma^{-3}, \quad \abs{\nabla\Gamma(z)\phi(z)}+\abs{\Gamma(z)\nabla\phi(z)}\leq N\gamma^{-4}.$$
		\item[(iii)] For any $z\in Q^+(1)$, we have
		$$\abs{\Gamma(z)\phi_t(z)}+\abs{\Gamma(z)\Delta\phi(z)}+\abs{\nabla\Gamma(z)\nabla\phi(z)}\leq N.$$
	\end{enumerate}
	Therefore, (\ref{eqn8c}) and Lemma \ref{lem1b} yield
	\begin{align*}
		\begin{split}
			A(\gamma)+E(\gamma)&=
			\gamma^{-1}\text{ess sup}_{-1\leq t\leq 0}\int_{B^+(\gamma)}\abs{u(t,x)}^2\,dx+\gamma^{-1}\int_{Q^+(\gamma)}\abs{\nabla u}^2 \,dz\\
			&\leq N\gamma^2\int_{Q^+}\abs{u}^2 \,dz +N\gamma^{-2} \int_{Q^+}(\abs{u}^2+2\abs{p-[p]_{1}})|u|\,dz\\
			&\leq N\gamma^2A(1)+N\gamma^{-2}(A(1)+E(1))^{3/2}+N(A(1)+E(1))^{1/2}D(1).
		\end{split}
	\end{align*}
	The lemma is proved.
\end{proof}

\begin{lemma}\label{lem5}
	Let $(u,p)$ be a pair of suitable weak solution of (\ref{NS1})-\eqref{NS3}. For $\gamma\in(0,1/2]$, $\rho>0$, and $\varepsilon_1\in(0,3/\tilde{r}-1)$, we have
	\begin{align}\label{eqn1c}
		D(\gamma\rho)\leq N\left[\gamma^{1+\varepsilon_1}(D_{\tilde q,1}(\rho)+(A(\rho)+E(\rho))^{\frac{1}{2}})+\gamma^{-1}(A(\rho)+E(\rho))\right],
	\end{align}
and
	\begin{align}\label{eqn1cc}
		D(\gamma\rho)\leq N\left[\gamma^{1+\varepsilon_1}(D(\rho)+(A(\rho)+E(\rho))^{\frac{1}{2}})+\gamma^{-1}(A(\rho)+E(\rho))\right],
	\end{align}
where $N$ is a constant independent of $u$, $p$, $\gamma$, and $\rho$, but may depend on $\tilde q$, $\tilde r$, and $\varepsilon_1$.
\end{lemma}

\begin{proof}
	 By the scale-invariant property, we may also assume $\rho=1$. We  fix a domain $\tilde{B}\subset \mathbb{R}^3$ with smooth boundary so that
	$$B^+(1/2)\subset \tilde{B}\subset B^+,$$
	and denote $\tilde{Q}=(-1,0)\times \tilde{B}$. Define $r^*$ by
	${1}/{r^*}={1}/{\tilde r}+{1}/{3}>1$. Using H\"{o}lder's inequality,  we get
	\begin{align*}
		\norm{u\cdot \nabla u}_{L_{\tilde{q}}^tL^x_{r^*}(Q^+)} \leq \norm{\nabla u}_{L_{2}^tL^x_{2}(Q^+)} \norm{ u}_{L_{q_1}^tL^x_{r_1}(Q^+)},
	\end{align*}
where $$\frac{1}{\tilde{q}} = \frac{1}{2}+\frac{1}{q_1},\quad \frac{1}{r^*} = \frac{1}{2}+\frac{1}{r_1}.$$
Because of the conditions on $(\tilde{q},\tilde{r})$, we have $q_1\in (2,\infty)$ and $2/q_1+3/r_1=3/2$. Thus by Lemma \ref{lem1b}, we get
\begin{equation}
\label{eqn5c}
\norm{u\cdot \nabla u}_{L_{\tilde{q}}^tL^x_{r^*}(Q^+)} \leq N(A(1)+E(1)).
\end{equation}
By the solvability of the linear Stokes system with the zero Dirichlet boundary condition  \cite[Theorem 1.1]{Refer29} (see Lemma \ref{lem2} below), there is a unique solution
	$$v\in W^{1,2}_{\tilde{q},r^*}(\tilde{Q}) \quad \text{and}\quad p_1\in W^{0,1}_{\tilde{q},r^*}(\tilde{Q})$$
	to the following initial boundary value problem:
	\[\begin{cases}
	\partial_t v-\Delta v+\nabla p_1=-u\cdot \nabla u \quad &\text{in } \tilde{Q},\\
	\nabla\cdot v=0 &\text{in } \tilde{Q},\\
	[p_1]_{\tilde{B}}(t)=0 & \text{for a.e. }t\in [-1,0],\\
	v=0 & \text{on }\partial_p\tilde{Q}.
	\end{cases}\]
	Moreover, we have
	\begin{align}\label{eqn6c}
		\begin{split}
			&\Vert |v|+|\nabla v|+|p_1|+|\nabla p_1| \Vert_{L^t_{\tilde{q}}L_{r^*}^x (\tilde{Q})}\le N\|u\cdot\nabla u\|_{L^t_{\tilde{q}}L_{r^*}^x (\tilde{Q})}
			\leq N(A(1)+E(1)).
		\end{split}
	\end{align}
	where in the last inequality we used (\ref{eqn5c}).
	By the Sobolev-Poincar\'{e} inequality, we have
	\begin{equation}
	\label{eqn7c}
	\norm{p_1}_{L_{\tilde{q}}^tL^x_{\tilde{r}}(Q^+)} \leq N\norm{\nabla p_1}_{L_{\tilde{q}}^tL^x_{r^*}(Q^+)} \leq N(A(1)+E(1)).
	\end{equation}
	We set $w=u-v$ and $p_2=p-p_1$. Then $w$ and $p_2$ satisfy
	\[\begin{cases}
	\partial_t w-\Delta w+\nabla p_2=0 \quad &\text{in } \tilde{Q},\\
	\nabla\cdot w=0 & \text{in } \tilde{Q},\\
	w=0 & \text{on } [-1,0) \times (\partial\tilde{B}\cap\{x_3=0\}).\end{cases}\]
    By the Sobolev-Poincar\'{e} inequality, H\"older's inequality, and an improved integrability result for the Stokes system \cite[Theorem 1.2]{Refer26b} (see also Lemma \ref{lem3} below),  we have
    \begin{align}
    	\label{eqn4c}
    	&\frac{1}{\gamma}\norm{ p_2-[p_2]_\gamma}_{L_{\tilde{q}}^tL^x_{\tilde r}(Q^+(\gamma))}
    	  \leq N \norm{\nabla p_2}_{L_{\tilde{q}}^tL^x_{\tilde r}(Q^+(\gamma))} \nonumber \\
    	  &\leq N \gamma^{3(\frac{1}{\tilde{r}}-\frac{1}{r'})}\norm{\nabla p_2}_{L_{\tilde{q}}^tL^x_{r'}(Q^+(\gamma))}
    	   \leq N \gamma^{3(\frac{1}{\tilde{r}}-\frac{1}{r'})}\norm{\nabla p_2}_{L_{\tilde{q}}^tL^x_{r'}(Q^+(1/2))}\nonumber \\
    	    &\leq N\gamma^{3(\frac{1}{\tilde{r}}-\frac{1}{r'})}\left(\norm{p_2-[p_2]_1}_{L_{\tilde{q}}^tL^x_{1}(Q^+)}+\norm{w}_{L_{\tilde{q}}^tL^x_{1}(Q^+)}\right)
     \nonumber \\
    	    &\leq  N\gamma^{3(\frac{1}{\tilde{r}}-\frac{1}{r'})} \left( D_{\tilde{q},1}(1) + \norm{p_1}_{L_{\tilde{q}}^tL^x_{1}(Q^+)}+ \norm{v}_{L_{\tilde{q}}^tL^x_{1}(Q^+)}+\norm{u}_{L_{\tilde{q}}^tL^x_{1}(Q^+)}\right) \nonumber \\
    	    &\leq  N\gamma^{3(\frac{1}{\tilde{r}}-\frac{1}{r'})} \left( D_{\tilde{q},1}(1) + A(1)+E(1)+A^{\frac{1}{2}}(1)\right).
          \end{align}
      where $r'>\tilde{r}$ is any large number and we also used \eqref{eqn6c} in the last inequality. From \eqref{eqn7c} and \eqref{eqn4c}, we obtain
      \begin{equation}
      D(\gamma)\leq N \left[ \gamma^{3(\frac{1}{\tilde{r}}-\frac{1}{r'})} \left( D_{\tilde{q},1}(1) + (A(1)+E(1))^{\frac{1}{2}}\right)+\gamma^{-1}(A(1)+E(1))\right]. \nonumber
      \end{equation}
 Because $\tilde{r}<3$ and $r'$ can be arbitrarily large, $\varepsilon_1$ may range in $(0,3/\tilde{r}-1)$. Finally, \eqref{eqn1cc} follows from \eqref{eqn1c} and H\"older's inequality. The lemma is proved.
\end{proof}

The following is the key lemma for the proof of Theorem \ref{thm0}.

\begin{lemma}
	\label{lem6}
	There exists a constant $\varepsilon_2=\varepsilon_2(\tilde{q})>0$ satisfying the following property. If
	\begin{equation*}
	A(1)+E(1)+D_{\tilde q,1}(1)\leq \varepsilon_2,
	\end{equation*}
	then there exist sufficiently small $\delta>0$ and $\rho_0>0$ such that
	\begin{equation}
	\label{eqn11c}
	A(\rho)+E(\rho)\leq \rho^{2-\delta},\quad D(\rho)\leq \rho^{1+\delta} 
	\end{equation}
	for any $\rho\in(0,\rho_0]$.
\end{lemma}
\begin{proof}
	Let $\rho_k = \rho_0^{(1+\beta)^k}$ where $\rho_0\in (0,1)$ and  $\beta\in (0,\delta)$ are small constants to be specified later. To prove \eqref{eqn11c} for any $\rho\in(0,1/2]$ (with a slightly smaller $\delta$), it suffices to show that, for every $k$ we have
	\begin{equation}
		\label{eqn12c}
		A(\rho_k)+E(\rho_k)\leq \rho_k^{2-\delta},\quad D(\rho_k)\leq \rho_k^{1+\delta}.
	\end{equation}
	We will assume the initial step for induction later by specify some conditions on $\rho_0$ and  $\varepsilon_2$. Now suppose $\eqref{eqn12c}$ is true for 0 to $k\geq l$ where $l$ is an integer to be specified later.  For  $E(\rho_{k+1})$,
we set $\gamma = \rho_k^{\beta}$ and $\rho = \rho_k$ in Lemma \ref{lem4} to obtain
\begin{align}
\label{eqn13c}
A(\rho_{k+1})+E(\rho_{k+1})&\leq N\left[ \rho_k^{2\beta+2-\delta}+\rho_k^{-2\beta+3(2-\delta)/2}+\rho_k^{-2\beta+1+\delta+(2-\delta)/2} \right] \nonumber\\
&\leq N\left[ \rho_k^{2\beta+2-\delta}+\rho_k^{-2\beta+3-3\delta/2}+\rho_k^{-2\beta+2+\delta/2}  \right].
\end{align}
For $\delta\ll 1$, when $\beta<\frac{3\delta}{2(4-\delta)}$, it holds that
\begin{equation}
\label{eqn14c}
\min\left\{2\beta+2-\delta,-2\beta+3-3\delta/2, -2\beta+2+\delta/2 \right\}>(2-\delta)(1+\beta).
\end{equation}
For $D(\rho_{k+1})$, we let $\tilde{\beta} = (1+\beta)^{l+1}-1$. We set $\gamma = \rho_{k-l}^{\tilde{\beta}}$ and $\rho = \rho_{k-l}$ in \eqref{eqn1c} to have
\begin{equation}
\label{eqn15c}
D(\rho_{k+1})\leq N\left[ \rho_{k-l}^{(1+\varepsilon_1)\tilde{\beta}+1+\delta} +\rho_{k-l}^{(1+\varepsilon_1)\tilde{\beta}+1-\delta/2}+\rho_{k-l}^{-\tilde{\beta}+2-\delta}\right].
\end{equation}
To satisfy
\begin{equation}
\label{eqn16c}
\min\left\{(1+\varepsilon_1)\tilde{\beta}+1+\delta, (1+\varepsilon_1)\tilde{\beta}+1-\delta/2,-\tilde{\beta}+2-\delta \right\}>(1+\delta)(1+\tilde{\beta}),
\end{equation}
we require $(\varepsilon_1-\delta)\tilde{\beta}>3\delta/2$. Indeed we can choose $\delta = \varepsilon_1^2$, $\beta = \varepsilon_1^2/4$, and take a sufficiently large integer $l$ of order $1/\varepsilon_1$ so that $\tilde \beta\sim 3\varepsilon_1$. A direct calculation shows that both \eqref{eqn14c} and \eqref{eqn16c} are satisfied. From \eqref{eqn13c} and \eqref{eqn15c} we can  find some $\xi>0$ such that
$$A(\rho_{k+1})+E(\rho_{k+1})\leq N\rho_{k+1}^{2-\delta+\xi},\quad D(\rho_{k+1})\leq N\rho_{k+1}^{1+\delta+\xi},$$
where $N$ is a constant independent of $k$ and $\xi$.
We choose $\rho_0$ small enough such that $$
N\rho_{k+1}^{\xi}<N\rho_0^{\xi}<1
$$
to get \eqref{eqn12c} for $k\ge l$.

Finally, by using \eqref{eqn1c} in lemma \ref{lem5} with $\gamma=1/2$ and $\rho=1$, we have
$$
A(1/2)+E(1/2)+D(1/2)\le N\varepsilon_2^{1/2}.
$$
Then by choosing $\varepsilon_2$ sufficiently small, we can make \eqref{eqn12c} to be true for $k = 0,1,2,\ldots,l$. Therefore, by induction we conclude  \eqref{eqn12c} is true for any integer $k\geq 0$.
\end{proof}
\begin{lemma}\label{lem7} Suppose $f(\rho_{0})\leq C_0$. If there exist $\alpha >\beta >0$ and $C_1,C_2>0$ such that for any $0<\gamma<\rho\le\rho_{0}$, it holds that
	$$
	f(\gamma)\leq C_1\left({\gamma}/{\rho}\right)^{\alpha}f(\rho)+C_2\rho^{\beta},
	$$
	then there exist constants $ C_3,C_4>0$ depending on  $C_0,C_1,C_2,\alpha,\beta$, such that
	$$
	f(\gamma)\leq C_3\left({\gamma}/{\rho_0}\right)^{\beta}f(\rho_0)+ C_4\gamma^{\beta}
	$$
	for any $\gamma\in (0,\rho_{0}]$.
\end{lemma}
\begin{proof}
See, for instance, \cite[Chapter III, Lemma 2.1]{iter}.
\end{proof}

Now we are ready to give
\begin{proof}[Proof of Theorem \ref{thm0}]
By Lemma \ref{lem6} we have the following estimates for $\rho\in (0,\rho_0]$:
\begin{equation}
\label{eqn17c}
\norm{p-[p]}_{L^t_{\tilde{q}}L^x_{\tilde{r}}(Q^+(\rho))}\leq \rho^{2+\delta},\quad
\norm{\abs{u}^2}_{L^t_{\tilde{q}}L^x_{\tilde{r}}(Q^+(\rho))}\leq \rho^{3-\delta}.
\end{equation}

Let $w$ be the unique weak solution to the heat equation
$$\partial_t w_i-\Delta w_i=-\partial_j(u_iu_j)-\partial_i(p-[p]_{\rho}) \quad \text{in } Q^+(\rho)$$
with the zero boundary condition. By the classical $L_p$ estimate for the heat equation, we have
$$\Vert\nabla w\Vert_{L^t_{\tilde{q}}L^x_{\tilde{r}}(Q^+(\rho))}\leq N\left\lVert |u|^2 \right\rVert_{L^t_{\tilde{q}}L^x_{\tilde{r}}(Q^+(\rho))} + N\left\lVert p-[p]_{\rho}\right\rVert_{L^t_{\tilde{q}}L^x_{\tilde{r}}(Q^+(\rho))} ,$$
which together with (\ref{eqn17c})  yields
\begin{equation}\label{eqn17d}
\Vert\nabla w\Vert_{L^t_{\tilde{q}}L^x_{\tilde{r}}(Q^+(\rho))}\leq  N\rho^{2+\delta}.
\end{equation}
By the Poincar\'{e} inequality with zero boundary condition, we get from \eqref{eqn17d} that
\begin{equation}\label{eqn19c}
\Vert w\Vert_{L^t_{\tilde{q}}L^x_{\tilde{r}}(Q^+(\rho))}\leq  N\rho^{3+\delta}.
\end{equation}

Denote $v=u-w$, which satisfies the homogeneous heat equation
$$\partial_t v-\Delta v=0 \quad \text{in }  Q^+(\rho)$$
with the boundary condition $v=u$ on $\partial_p Q^+(\rho)$.
Let $0<\gamma<\rho$. By the Poincar\'{e} inequality with zero mean value  and using the fact that any H\"older norm of a caloric function in a smaller half  cylinder is controlled by any $L_p$ norm of it in a larger half cylinder. We have
\begin{align}\label{eqn188c}
	\begin{split}
	\norm{v-(v)_\gamma}_{L^t_{\tilde{q}}L^x_{\tilde{r}}(Q^+(\gamma))}
		&\leq N\gamma^{4}\norm{v}_{C^{1/2,1}(Q^+(\gamma))}\\
		&\leq N(\gamma/\rho)^{4}\norm{v-(u)_{\rho}}_{L^t_{\tilde{q}}L^x_{\tilde{r}}(Q^+(\rho))},
	\end{split}
\end{align}
where $(u)_{\rho}$ is the average of $u$ in $Q^+(\rho)$.

Using (\ref{eqn188c}), (\ref{eqn19c}), and the triangle inequality, we have
\begin{align*}
	\Vert u-(u)_{\gamma}\Vert_{L^t_{\tilde{q}}L^x_{\tilde{r}}(Q^+(\gamma))} \leq &\Vert v-(v)_{\gamma}\Vert_{L^t_{\tilde{q}}L^x_{\tilde{r}}(Q^+(\gamma))}+2\Vert w\Vert_{L^t_{\tilde{q}}L^x_{\tilde{r}}(Q^+(\gamma))}\\
	 \leq & N(\gamma/\rho)^{4} \Vert v-(u)_{\rho}\Vert_{L^t_{\tilde{q}}L^x_{\tilde{r}}(Q^+(\rho))} + N \rho^{3+\delta}\\
	\leq & N(\gamma/\rho)^{4}
	\left(\Vert u-(u)_{\rho}\Vert_{L^t_{\tilde{q}}L^x_{\tilde{r}}(Q^+(\rho))}
	+\Vert w\Vert_{L^t_{\tilde{q}}L^x_{\tilde{r}}(Q^+(\rho))}\right)  + N \rho^{3+\delta}\\
	 \leq & N(\gamma/\rho)^{4}\Vert u-(u)_{\rho}\Vert_{L^t_{\tilde{q}}L^x_{\tilde{r}}(Q^+(\rho))} +N \rho^{3+\delta}.
\end{align*}
Applying Lemma \ref{lem7} we obtain
$$\Vert u-(u)_{\gamma}\Vert_{L^t_{\tilde{q}}L^x_{\tilde{r}}(Q^+(\gamma))}\leq N \gamma^{3+\delta} $$
for any $\gamma\in (0,\rho_0)$. By H\"{o}lder's inequality, we have
\begin{equation*}
\Vert u-(u)_{\gamma}\Vert_{L_1(Q^+(\gamma))}\leq N\gamma^2
\Vert u-(u)_{\gamma}\Vert_{L^t_{\tilde{q}}L^x_{\tilde{r}}(Q^+(\gamma))}\leq N \gamma^{5+\delta}.
\end{equation*}
Similar estimates can be derived for interior points using  same techniques.
We conclude that $u$ is H\"{o}lder continuous in $Q^+(\rho_0)$ by Campanato's characterization of H\"{o}lder continuity. The theorem then follows by a covering argument.
\end{proof}

\section{A boundary regularity criterion without involving $\nabla u$}         \label{sec4}
In this section, we shall prove the following theorem, which is more general than Theorem \ref{thm4}.
\begin{theorem}
	\label{thm1b}
For each triple of exponents  $(q,r,\tilde{q})$ satisfying
\begin{equation}
                            \label{eq9.15}
	\frac{2}{q}+\frac{3}{r}<2,\quad 2<q<\infty,\quad \frac{3}{2}<r<\infty,
\end{equation}
	and
\begin{equation}
	\label{eqn100b}
\frac{1}{\tilde{q}}<F(q,r):=1-
\frac{1}{q}\cdot\frac{\left(1/r-1/2\right)_+}{1/r-1/6},
\end{equation}
there exists a universal constant $\e_0 = \e_0(q,r,\tilde{q})$  such that if  $(u,p)$ is a pair of suitable weak solution to \eqref{NS1}-\eqref{NS3} in $Q^+$ with $p\in L_{\tilde q}^tL^x_1(Q^+)$ and satisfies
	\begin{equation}
	\label{eqn0b}
	C_{q,r}(1)+D_{\tilde{q},1}(1)<\varepsilon_0,
	\end{equation}
	then $u$ is regular in $\overline{Q^+(1/2)}$.
\end{theorem}

\begin{remark}
The restriction \eqref{eqn100b} arises in the estimates for the pressure term below by using the coercive estimate for the linear Stokes system. It is not clear to us if this restriction can be dropped.
A straightforward calculation shows that under the constraints
$$
	\frac{2}{q}+\frac{3}{r}\le 2,\quad 2<q<\infty,\quad \frac{3}{2}<r<\infty,
$$
the function $F(q,r)$ attains its minimum $1-(\sqrt 3-\sqrt 2)^2/4$ when $1/r=1/\sqrt 6+1/6$ and $1/q=3/4-\sqrt 6/4$. Therefore, if
$$
\tilde q\ge \left(1-\frac {(\sqrt 3-\sqrt 2)^2} 4\right)^{-1}\approx 1.02591, $$
then \eqref{eqn100b} is trivial for any $(q,r)$ satisfying \eqref{eq9.15}. Moreover, it is easily seen that
$$
\max\{1-1/q,1/2+1/q\}< F(q,r).
$$
Therefore, by decreasing $\tilde q$ if necessary, we may assume that
\begin{equation}
                            \label{eqn111}
\max\{1-1/q,1/2+1/q,7/8\}<1/\tilde q< F(q,r).
\end{equation}
\end{remark}

\begin{remark}
By a slight modification of the proof below, we have the following result. Under the conditions of Theorem \ref{thm1b}, if instead of \eqref{eqn0b} we assume
$$
	C_{q,r}(1)+D_{\tilde{q},1}(1)\le C_0
$$
for some constant $C_0>0$, then
$$
	A(1/2)+E(1/2)\le N(q,r,\tilde q, C_0)<\infty.
$$
\end{remark}

We first set up some notation and state a few lemmas that  will be useful in the proof of Theorem \ref{thm1b}.
Let $\rho_{k}=1-2^{-k}$. Denote $Q^+_{k}=Q^+(\rho_k)$ and $B^+_k = B^+(\rho_k)$ for integer $k\geq 1$.  Again we denote $A_k = A(\rho_k)$ and $E_k = E(\rho_k)$. For each integer $k\geq 0$, We  fix a domain $\tilde{B}_k\subset \mathbb{R}^3$ with smooth boundary so that
$$
B^+_{k+1}\subset\tilde{B}_k \subset B^+_{k+2}
$$
such that the $C^2$ norm of $\partial \tilde B_k$ is bounded by $N2^{k}$.
We also denote $\tilde{Q}_k=(-\rho^2_{k+1},0)\times \tilde{B}_k$.

\begin{lemma}\label{lem2}
	Let $m,n\in (1,\infty)$ be two fixed numbers. Assume that $g\in L_{n,m}(\tilde{Q}_k).$ Then there exists a unique function pair $(v,p)$, which satisfies the following equations:
	\[\begin{cases}
	\partial_t v-\Delta v+\nabla p=g \quad &\text{in } \tilde{Q}_k,\\
	\nabla\cdot v=0 &\text{in }\tilde{Q}_k,\\
	[p]_{\tilde{B}_k}(t)=0 & \text{for a.e. }t\in [-\rho^2_{k+1},0],\\
	v=0 & \text{on }\partial_p \tilde{Q}_k.
	\end{cases}\]
	Moreover, $v$ and $p$ satisfy the following estimate:
	$$\Vert v\Vert _{W^{1,2}_{n,m}(\tilde{Q}_k)}+\Vert p\Vert _{W^{0,1}_{n,m}(\tilde{Q}_k)}\leq C2^{bk} \Vert g\Vert _{L_{n,m}(\tilde{Q}_k)},$$
	where the constants $C$ and $b$ only depend on $m$ and $n$.
\end{lemma}
We refer the reader to \cite[Theorem 1.1]{Refer29} for the proof of Lemma \ref{lem2}. The factor $2^{bk}$ can be obtained by keeping track of the constants in the localization argument in \cite[Sect. 3]{Refer29}.

\begin{lemma}\label{lem3}
	Let $n,s\in (1,\infty)$ be constants and $g\in L_{n}^tL_s^x(Q^+_{k+1})$. Assume that the functions $v\in W^{0,1}_{n,1}(Q^+_{k+1})$ and $p\in L_{n}^tL_1^x(Q^+_{k+1})$ satisfy the equations:
	\[\begin{cases}
	\partial_tv-\Delta v+\nabla p=g \quad &\text{in } Q^+_{k+1},\\
	\nabla\cdot v=0 &\text{in }Q^+_{k+1},
	\end{cases}\]
	and the boundary condition
	$$ v=0 \quad \text{on }\{y\,\vert\, y=(y',0),|y'|<\rho_{k+1}\}\times[-\rho^2_{k+1},0).$$
	Then we have $v\in W^{1,2}_{n,s}(Q^+_k),\, p\in W^{0,1}_{n,s}(Q^+_k)$, and
	\begin{equation}
                        \label{eq9.31}
		\Vert v\Vert _{W^{1,2}_{n,s}(Q^+_k)}+\Vert p\Vert _{W^{0,1}_{n,s}(Q^+_k)}\leq C2^{bk} \left(\Vert g\Vert _{L_{n}^t L_s^x(Q^+_{k+1})}+\Vert v\Vert _{L_{n}^t L_1^x(Q^+_{k+1})}+\Vert p\Vert _{L_{n}^t L_1^x(Q^+_{k+1})}\right),
	\end{equation}
	where the constants $C$ and $b$ only depend on $n$ and $s$.
\end{lemma}
\begin{proof}
We use a mollification argument. Denote $x'=(x_1,x_2)$ and $\hat Q_k=Q^+((\rho_k+\rho_{k+1})/2)$.
By the Sobolev embedding theorem, we have $v\in L_{n}^t L_m^x(Q_{k+1}^+)$ for some $m\in (1,\min(n,s))$. Let $v^{\varepsilon}$, $p^{\varepsilon}$, and $g^{\varepsilon}$ be the standard mollifications with respect to $(t,x')$, which satisfy the same equations as $v$, $p$, and $g$. By the properties of mollifications, it is clear that for sufficiently small $\varepsilon$,
\begin{equation}
                                    \label{eq2.25}
D_{x'} v^{\varepsilon},\,\,D_{x'}^2 v^{\varepsilon},\,\, \partial_t v^{\varepsilon}\in L_{n}^t L_m^x(\hat Q_k),\quad D_{x'}p^{\varepsilon}\in L_{n}^{t,x'} L_1^{x_3}(\hat Q_k),\quad j=1,2.
\end{equation}
Then from the equations for $v_1^{\varepsilon}$ and $v_2^{\varepsilon}$, we get $D_{x_3x_3}v_j^{\varepsilon}\in L_{n}^{t,x'} L_1^{x_3}(\hat Q_k)$ for $j=1,2$. By applying the Sobolev embedding theorem in the $x_3$ direction, we get $D_{x_3}v_j^{\varepsilon}\in L_{n}^t L_m^x(\hat Q_k)$, which together with \eqref{eq2.25} implies that $Dv_j^{\varepsilon}\in L_{n}^t L_m^x(\hat Q_k)$ for $j=1,2$. Since $D_{x'}v^{\varepsilon}$ satisfies the same equation, we see that $D_{x'}Dv_j^{\varepsilon}\in L_{n}^t L_m^x(\hat Q_k)$ for $j=1,2$. Now
owing to $\nabla\cdot v^{\varepsilon}=0$, we have
$$
D_{x_3}v_3^{\varepsilon},\ D_{x_3x_3}v_3^{\varepsilon}\in L_{n}^t L_m^x(\hat Q_k),
$$
which together with the equation for $v_3^{\varepsilon}$ further implies $D_{x_3}p^{\varepsilon}\in L_{n}^t L_m^x(\hat Q_k)$. Using the Sobolev embedding theorem, we then get $p^{\varepsilon}\in L_{n}^t L_m^x(\hat Q_k)$. By \cite[Theorem 1.2]{Refer26b} (see also \cite{Refer24b}), we have
$v^{\varepsilon}\in W^{1,2}_{n,s}(\hat Q_k)$, $p^{\varepsilon}\in W^{0,1}_{n,s}(\hat Q_k)$, and
\begin{equation*}
\|v^{\varepsilon}\|_{W_{n,s}^{1,2}
(Q^+_k)}+\|p^{\varepsilon}\|_{W_{n,s}^{0,1}
(Q^+_k)}
\leq C2^{bk}\big(\|g^{\varepsilon}\|_{L_{n}^t L_s^x(\hat Q_k)}
+\|v^{\varepsilon}\|_{W_{n,m}^{1,0}(\hat Q_k)}
+\|p^{\varepsilon}\|_{L_{n}^t L_m^x(\hat Q_k)}\big),
\end{equation*}
where $C$ is independent of $\varepsilon$ and $k$. Again the factor $2^{bk}$ can be obtained by keeping track of the constants in the proofs in \cite{Refer24b}. Taking the limit as $\varepsilon\to 0$, we get
\begin{equation*}
\|v\|_{W_{n,s}^{1,2}
(Q^+_k)}+\|p\|_{W_{n,s}^{0,1}
(Q^+_k)}\leq C2^{bk}\big(\|g\|_{L_{n}^t L_s^x(\hat Q_k)}
+\|v\|_{W_{n,m}^{1,0}(\hat Q_k)}
+\|p\|_{L_{n}^t L_m^x(\hat Q_k)}\big).
\end{equation*}
By interpolation inequalities, for any $\delta\in(0,1)$ we have
\begin{align*}
&\|v\|_{W_{n,s}^{1,2}
(Q^+_k)}+\|p\|_{W_{n,s}^{0,1}
(Q^+_k)}\\
&\, \leq \delta \big(\|v\|_{W_{n,s}^{1,2}
(\hat Q_k)}+\|p\|_{W_{n,s}^{0,1}
(\hat Q_k)}\big)+
C2^{bk}\|g\|_{L_{n}^t L_s^x(\hat Q_k)}
+C_\delta 2^{bk}\big(\|v\|_{L_{n}^x L_1^t (\hat Q_k)}
+\|p\|_{L_{n}^t L_1^x(\hat Q_k)}\big).
\end{align*}
Finally, \eqref{eq9.31} follows by using a standard iteration argument.
\end{proof}

We now give the proof of Theorem \ref{thm1b}.	
	\begin{proof}[Proof of Theorem \ref{thm1b}]
By replacing $p$ with $p-[p]_1(t)$ without loss of generality, we may assume that $[p]_1(t)=0$ for $t\in (-1,0)$.
Let $\tau>0$ be such that
\begin{equation*}
\frac{2}{q}+\frac{3}{r} : =2-\tau.
\end{equation*}
It is easily seen that we can find $\hat r\in (3/2,r)$ such that
\begin{equation}
                                    \label{eq4.54}
\frac{2}{q}+\frac{3}{\hat{r}} : =2-\tau_1<2,\quad
\frac{1}q+\frac{2}{\hat{r}} > 1, \quad \frac{1}q+\frac{1}{\hat{r}} > \frac 2 3.
\end{equation}
In the sequel, we will choose $\tau_1$ sufficiently small by reducing $\hat{r}$.

Let $\rho_{k}=1-2^{-k}$, $B_k = B(\rho_k)$, and $Q_{k}=Q(\rho_k)$ for integer $k\geq 1$.
For each $k$, we choose a cut-off function $\psi_k$ satisfying
$$ \text{supp } \psi_k \subset Q_{k+1}, \quad \psi_k\equiv 1 \text{ on } Q_{k},$$
$$\partial_t \psi_k\leq N2^{2k}, \quad D\psi_k\leq N2^k, \quad D^2\psi_k\leq N2^{2k}.$$
By \eqref{eqn_sw_energy}, we have
\begin{equation*}
		\text{ess sup}_{-\rho_k^2 \leq s\leq 0}\int_{B^+_k}\abs{u(s,x)}^2\,dx+2\int_{Q^+_k}\abs{\nabla u}^2 \,dx \,ds\\
		\leq N\int_{Q^+_{k+1}}\{2^{2k}\abs{u}^2+2^k(\abs{u}^2+2|p|)|u|\}\,dx \,ds,
\end{equation*}
where $N$ is independent of $k$.
For simplicity, we denote $A_k = A(\rho_k)$ and $E_k = E(\rho_k)$. Thus we can rewrite the above inequality as
\begin{equation}
\label{eqn1}
A_k+E_k \leq N 2^{2k}\int_{Q^+_{k+1}}\left\{ \abs{u}^2+\abs{u}^3+\abs{pu} \right\}\, dz.
\end{equation}
We have the following interpolation using H\"{o}lder's inequality
\begin{equation}
\int_{Q^+_{k+1}} \abs{u}^3\,dz\leq \norm{u}^{\frac{1}{1-2\tau_1}}_{L_{q}^tL_{\hat r}^x(Q^+_{k+1})}\norm{u}^{\frac{2-6\tau_1}{1-2\tau_1}}_{L_{q_1}^tL_{r_1}^x(Q^+_{k+1})}\leq \norm{u}^{\frac{1}{1-2\tau_1}}_{L_{q}^tL_{r}^x(Q^+_{k+1})}
\norm{u}^{\frac{2-6\tau_1}{1-2\tau_1}}_{L_{q_1}^tL_{r_1}^x(Q^+_{k+1})}, \nonumber
\end{equation}
where
\begin{equation}
                                        \label{eq5.04}
q_1\in [2,\infty],\quad r_1\in [2,6]
\end{equation}
need to satisfy $2/q_1+3/r_1=3/2$. A simple calculation shows that
$$
q_1 = \frac{(2-6\tau_1)q}{(1-2\tau_1)q-1}, \quad r_1 = \frac{(2-6\tau_1)\hat r}{(1-2\tau_1)\hat r-1}.
$$
Indeed \eqref{eq5.04} holds because of \eqref{eq4.54}. We then apply Lemma \ref{lem1b} to get
\begin{equation}
\label{eqn2}
\int_{Q^+_{k+1}} \abs{u}^3\,dz\leq \norm{u}^{\frac{1}{1-2\tau_1}}_{L_{q}^tL_{r}^x(Q^+_{k+1})}(A_{k+1}+E_{k+1})^{\frac{1-3\tau_1}{1-2\tau_1}}.
\end{equation}
Again by H\"{o}lder's inequality, we have
\begin{equation}
\label{eqn13}
\int_{Q^+_{k+1}} |u|^2\,dz
\leq N\left(\int_{Q^+_{k+1}} \abs{u}^3\,dz\right)^{2/3}
\leq N\norm{u}^{\frac{2}{3(1-2\tau_1)}}_{L_q^tL_r^x(Q^+_{k+1})}
(A_{k+1}+E_{k+1})^{\frac{2(1-3\tau_1)}{3(1-2\tau_1)}}.
\end{equation}

To deal with the last term in \eqref{eqn1}, we make the following decomposition. For some suitable $(q',s^*)$ which we will specify later, there exists a pair of unique solution
	 $$v_k\in W^{1,2}_{q',s^*}(\tilde{Q}_k) \quad \text{and}\quad p_k\in W^{0,1}_{q',s^*}(\tilde{Q}_k)$$
	to the following initial boundary value problem:
	\begin{equation}
	\label{eqn_decomp1}
	\begin{cases}
	\partial_t v_k-\Delta v_k+\nabla p_k=-u\cdot \nabla u \quad &\text{in } \tilde{Q}_k,\\
	\nabla\cdot v_k=0 &\text{in } \tilde{Q}_k,\\
	[p_k]_{\tilde{B}_k}(t)=0 & \text{for a.e. }t\in [-\rho^2_{k+1},0],\\
	v_k=0 & \text{on }\partial_p\tilde{Q}_k.
	\end{cases}\end{equation}
	We set $w_k=u-v_k$ and $h_k=p-p_k$. Then $w_k$ and $h_k$ satisfy
	\begin{equation*}
	\begin{cases}
	\partial_t w_k-\Delta w_k+\nabla h_k=0 \quad &\text{in } \tilde{Q}_k,\\
	\nabla\cdot w_k=0 & \text{in } \tilde{Q}_k,\\
	w_k=0 & \text{on } [-\rho^2_{k+1},0) \times (\partial\tilde{B}_k\cap \{x_3=0  \}).\end{cases}
	\end{equation*}
We choose $\hat{r}_2\in (3/2,r)$ satisfying
$$
\frac{2}{q}+\frac{3}{\hat{r}_2} : =2-\tau_2<2.
$$
In the sequel, we will choose $\tau_2$ sufficiently small by reducing $\hat{r}_2$.

{\em Estimates for $p_k$:}
 For $\alpha_1\in (0,1)$ which we will specify later, by H\"older's inequality,
	\begin{equation}
	\label{eqn11b}
	\int_{Q^+_{k+1}}\abs{up_k}\, dz\leq \norm{u}^{1-\alpha_1}_{L_{q_2}^tL^x_{r_2}(Q^+_{k+1})}
	\norm{u}^{\alpha_1}_{L_{q}^tL^x_{\hat{r}_2}(Q^+_{k+1})}
	\norm{p_k}_{L^t_{q'}L^x_{s'}(Q^+_{k+1})},
	\end{equation}
	where the exponents need to satisfy
	$$
\frac{2}{q_2}+\frac{3}{r_2} = \frac{3}{2},\quad \frac{1}{q'}+\frac{1-\alpha_1}{q_2}+\frac{\alpha_1}{q}=1,\quad
\frac{1}{s'}+\frac{1-\alpha_1}{r_2}+\frac{\alpha_1}{\hat{r}_2}=1.
$$
	 The system above implies that
	\begin{equation}
	\label{eqn7b}
	\frac{2}{q'}+\frac{3}{s'} = 5-\left((1-\alpha_1)\cdot \frac{3}{2}+\alpha_1(2-\tau_2)\right).
	\end{equation}
	To make use of Lemma \ref{lem1b}, we need $ 2\leq q_2\leq \infty$, which is equivalent to
	\begin{equation}
	\label{eqn17b}
	\frac{1}{q'}+\frac{\alpha_1}q\leq 1\leq \frac{1}{q'}+\frac{\alpha_1}q+\frac{1-\alpha_1}{2}.
	\end{equation}
We are going to check this condition later.

Next we estimate the nonlinear term $u\cdot \nabla u$. Define another exponent $s^*$ by $$ \frac{1}{s^*} := \frac{1}{s'}+\frac{1}{3}. $$
	 Using H\"{o}lder's inequality,  we get
	\begin{equation}
	\label{eqn20}
	\norm{u\cdot \nabla u}_{L^t_{q'}L^x_{s^*}(\tilde{Q}_k)}\leq
	\norm{ \nabla u}_{L^t_2L^x_2(\tilde{Q}_k)}\norm{u}^{1-\alpha_2}_{L^t_{q}L^x_{\hat{r}_2}(\tilde{Q}_k)}\norm{u}^{\alpha_2}_{L^t_{q_3}L^x_{r_3}(\tilde{Q}_k)},
	\end{equation}
	where
\begin{equation}
                                        \label{eq5.00}
\frac{1}{q'}=\frac{1}{2}+\frac{1-\alpha_2}{q}+\frac{\alpha_2}{q_3},
\quad\frac{1}{s^*} = \frac{1}{2}+\frac{1-\alpha_2}{\hat{r}_2}+\frac{\alpha_2}{r_3}
\end{equation}
	$$q_3\in [2,\infty],\quad r_3\in [2,6],\quad \frac{2}{q_3}+\frac{3}{r_3} = \frac{3}{2}.$$
	In particular, we take $q_3 = 2$ and $r_3 = 6$ so that
	\begin{equation}
	\label{eqn28}
	\alpha_2 = \frac{{1}/{q'}-{1}/{2}-{1}/{q}}{{1}/{2}-{1}/{q}}.
	\end{equation}
	 By \eqref{eq5.00} we have
	\begin{equation}
	\label{eqn19b} \frac{2}{q'}+\frac{3}{s'}=(1-\alpha_2)(2-\tau_2)+(1+\alpha_2)\cdot\frac{3}{2}.
	\end{equation}
	From \eqref{eqn7b}, \eqref{eqn28}, and \eqref{eqn19b}, we get
	$$\alpha_1= \alpha_2+\frac{2\tau_2}{1-2\tau_2} = \frac{1/q'-1/2-{1}/{q}}{{1}/{2}-{1}/{q}}+\frac{2\tau_2}{1-2\tau_2}.$$
	Since we have solved for $\alpha_1$, we can now go back to verify \eqref{eqn17b}. A simple calculation gives that it indeed holds when $q'> \frac{q^2}{q^2-q+2}$ and $\tau_2$ is sufficiently small. We note that there is an implicit restriction $q>2$ contained in the  conditions above. We can observe it by  adding up the first inequality in \eqref{eqn17b} and the first equality in \eqref{eq5.00} and using the fact $\alpha_1>\alpha_2$.

To make use of Lemma \ref{lem2}, we need to check that $s^*> 1$, which is equivalent to
	$$\frac{1-\alpha_2}{\hat{r}_2}+\frac{\alpha_2}{6}< \frac{1}{2}.$$
In the special case when $q'=\frac{q^2}{q^2-q+2}$ , we have $\alpha_2 = 1-2/q\in (0,1)$ and the above inequality becomes
	\begin{equation*}
	\frac{2}{\hat{r}_2 q}+\left(1-\frac{2}{q}\right) \frac{1}{6}< \frac{1}{2},
	\end{equation*}
which clearly holds true because
	$$2\sqrt{\frac{6}{\hat{r}_2q}}\leq \frac{2}{q}+\frac{3}{\hat{r}_2}<2$$
and thus $\hat{r}_2 q> 6$.
	By continuity, when $q'$ is sufficiently close to $\frac{q^2}{q^2-q+2}$, we still have $s^*> 1$ and $\alpha_2\in (0,1)$.

Now by Lemma \ref{lem2}, we have the existence of the unique solution pair $(v_k,p_k)$ to \eqref{eqn_decomp1} and 
	 \begin{align*}
	 	\begin{split}
	 		&\Vert |p_k|+|\nabla p_k| \Vert_{L^t_{q'}L_{s^*}^x (\tilde{Q}_k)}\le N2^{bk} \|u\cdot\nabla u\|_{L^t_{q'}L_{s^*}^x (\tilde{Q}_k)}\\
	 		&\leq N2^{bk}\norm{ \nabla u}_{L^t_2L^x_2(\tilde{Q}_k)}\norm{u}^{1-\alpha_2}_{L^t_qL^x_{\hat r_2}(\tilde{Q}_k)}\norm{u}^{\alpha_2}_{L^t_{q_3}L^x_{r_3}(\tilde{Q}_k)},
	 	\end{split}
	 \end{align*}
	 where in the last inequality we used (\ref{eqn20}). Here and in the sequel, $b$ is a positive constant which is independent of $k$ and may vary from line to line.
      Together with  the Sobolev-Poincar\'{e} inequality and H\"{o}lder's inequality, we  obtain
     \begin{equation*}
     \norm{p_k}_{L^t_{q'}L^x_{s'}(\tilde{Q}_k)}\leq N2^{bk}\norm{ \nabla u}_{L^t_2L^x_2(\tilde{Q}_k)}\norm{u}^{1-\alpha_2}_{L^t_{q}L^x_{\hat r_2}(\tilde{Q}_k)}\norm{u}^{\alpha_2}_{L^t_{q_3}L^x_{r_3}(\tilde{Q}_k)}. \end{equation*}
Combining with \eqref{eqn11b}  we have
     \begin{align}
     	\label{eqn5p}
     	\int_{Q^+_{k+1}}\abs{up_k}\, dz  &\leq  N2^{b k} \norm{u}^{1-\alpha_1}_{L_{q_2}^tL^x_{r_2}(Q^+_{k+2})}
     	\norm{u}^{1+\alpha_1-\alpha_2}_{L^t_{q}L^x_{\hat r_2}(Q^+_{k+2})} \cdot\norm{ \nabla u}_{L^t_2L^x_2(Q^+_{k+2})}\norm{u}^{\alpha_2}_{L^t_{q_3}L^x_{r_3}(Q^+_{k+2})}\nonumber \\
     	& \leq N2^{b k}\epsilon_0^{\frac{1}{1-2\tau_2}}(A_{k+2}+E_{k+2})^{\frac{1-3\tau_2}{1-2\tau_2}}.
     \end{align}

{\em Estimates for $h_k$:} 
	 For some $\alpha_3\in (0,1]$ which we will specify later,  analogous to \eqref{eqn11b} by H\"older's inequality,
	 \begin{equation}
	 \label{eqn11h}
	 \int_{Q^+_{k+1}}\abs{uh_k}\, dz\leq \norm{u}^{1-\alpha_3}_{L_{q_4}^tL^x_{r_4}(Q^+_{k+1})}
	 \norm{u}^{\alpha_3}_{L_{q}^tL^x_{r}(Q^+_{k+1})}
	 \norm{h_k}_{L^t_{\tilde{q}}L^x_{\tilde{s}}(Q^+_{k+1})},
	 \end{equation}
where the exponents satisfy
	 $$\frac{2}{q_4}+\frac{3}{r_4} = \frac{3}{2},\quad
\frac{1}{\tilde{q}}+\frac{1-\alpha_3}{q_4}+\frac{\alpha_3}{q}=1,\quad
\frac{1}{\tilde{s}}+\frac{1-\alpha_3}{r_4}+\frac{\alpha_3}{r}=1.$$
	 $$$$
To make use of Lemma \ref{lem1b}, we require $ 2\leq q_4\leq \infty$, which is equivalent to
	 \begin{equation*}
	 \frac{1}{\tilde{q}}+\frac{\alpha_3}{q}\leq 1\leq \frac{1}{\tilde{q}}+\frac{\alpha_3}{q}+\frac{1-\alpha_3}{2}.
	 \end{equation*}
We simply choose $\alpha_3 = q(1-1/\tilde{q})\in (0,1)$ and use \eqref{eqn111} to verify this condition.

	 By Lemma \ref{lem3} and the triangle inequality, we have $h_k\in W^{0,1}_{\tilde{q},\tilde{s}}(Q^+_{k+2})$ and the following estimate:
	 \begin{align}
	 	\label{eqn29}
	 	&\Vert  h_k\Vert_{ L^t_{\tilde{q}}L^x_{\tilde{s}}(Q^+_{k+1})}\leq N2^{bk}\Vert |w_k|+|h_k|\Vert_{ L^t_{\tilde{q}}L^x_{1}(Q^+_{k+2})}\nonumber\\
	 	&\leq N 2^{bk}\Vert |u|+|p|\Vert_{ L^t_{\tilde{q}}L^x_{1}(Q^+_{k+2})}  +N2^{bk}\Vert |v_k|+|p_k|\Vert_{ L^t_{\tilde{q}}L^x_{1}(Q^+_{k+2})}.
	 \end{align}
 By H\"{o}lder's inequality, \eqref{eqn111}, and \eqref{eqn0b}, we have
 \begin{equation}
 \label{eqn26}
 \Vert |u|+|p|\Vert_{ L^t_{\tilde{q}}L^x_{1}(Q^+_{k+2})}\leq \varepsilon_0.
 \end{equation}
	 For any $\tilde{s}^*>1$, by Lemma \ref{lem2},  we have
	 \begin{equation}
	 	\label{eqn23b}
	 \norm{|v_k|+|p_k|}_{L^t_{\tilde{q}}L_{1}^x (\tilde{Q}_{k+1})} \le  N\norm{|v_k|+|p_k|}_{L^t_{\tilde{q}}L_{\tilde{s}^*}^x (\tilde{Q}_{k+1})}\le N2^{bk} \|u\cdot\nabla u\|_{L^t_{\tilde{q}}L_{\tilde{s}^*}^x (\tilde{Q}_{k+1})}.
	 \end{equation}
	 Next analogous to \eqref{eqn20} by H\"older's inequality,
	 \begin{equation}
	 \label{eqn20h}
	 \norm{u\cdot \nabla u}_{L^t_{\tilde{q}}L^x_{\tilde{s}^*}(\tilde{Q}_{k+1})}\leq
	 \norm{ \nabla u}_{L^t_2L^x_2(\tilde{Q}_{k+1})}
\norm{u}^{1-\alpha_4}_{L^t_qL^x_r(\tilde{Q}_{k+1})}\norm{u}^{\alpha_4}_{L^t_{q_5}L^x_{r_5}(\tilde{Q}_{k+1})},
	 \end{equation}
where the exponents satisfy
$$2\leq q_5\leq \infty,\quad\frac{2}{q_5}+\frac{3}{r_5} = \frac{3}{2},\quad
\frac{1}{2}+\frac{1-\alpha_4}{q}+\frac{\alpha_4}{q_5}\le \frac{1}{\tilde{q}},
\quad \frac{1}{2}+\frac{1-\alpha_4}{r}+\frac{\alpha_4}{r_5}\le \frac{1}{\tilde{s}^*}.
$$
To justify the use of Lemma \ref{lem2} in \eqref{eqn23b}, we need to check that $\tilde{s}^*> 1$, which is equivalent to
	 $$\frac{1-\alpha_4}{r}+\frac{\alpha_4}{r_5}<\frac{1}{2}.$$  
For later purpose, we also want $\alpha_3 = q(1-1/\tilde{q})>\alpha_4$. To satisfy all of the three  conditions above, we discuss two cases:
	
	 i) If $r>2$, we simply choose $\alpha_4=0$. Then such $\tilde s^*>1$ exists, and in view of \eqref{eqn111}, all the conditions are satisfied.
	
	 ii) If $r\in (3/2, 2]$, then $q>4$. We set $q_5 = 2, r_5 = 6$. To ensure $\tilde{s}^*>1$ and  $\alpha_3>\alpha_4$, we take
$$
\frac{1/r-1/2}{1/r-1/6}<\alpha_4< \min\{1/2, q(1-1/\tilde q)\},
$$
which is possible because of \eqref{eqn100b} and $r>3/2$. Moreover, we have
$$
\frac{1}{2}+\frac{1-\alpha_4}{q}+\frac{\alpha_4}{q_5}
\le \frac{1}{2}+\frac{1-\alpha_4}{4}+\frac{\alpha_4}{2}\le \frac 7 8<\frac 1 {\tilde q},
$$
where we used \eqref{eqn111} in the last inequality.

Now plugging  (\ref{eqn26}), \eqref{eqn23b},  and \eqref{eqn20h} into \eqref{eqn29},  we  obtain
	 \begin{equation*}
	 	\norm{h_k}_{L^t_{\tilde{q}}L^x_{\tilde{s}}(Q^+_{k+1})}\leq N2^{bk} \left(\varepsilon_0+ \norm{ \nabla u}_{L^t_2L^x_2(Q^+_{k+3})}\norm{u}^{1-\alpha_4}_{L^t_qL^x_r(Q^+_{k+3})}
\norm{u}^{\alpha_4}_{L^t_{q_5}L^x_{r_5}(Q^+_{k+3})} \right).
	 \end{equation*}
	 Together with \eqref{eqn11h}  we have
	 \begin{align}
	 	\label{eqn5b}
	 	\int_{Q^+_{k+1}}\abs{uh_k}\, dz  \leq & N2^{b k} \norm{u}^{1-\alpha_3}_{L_{q_4}^tL^x_{r_4}(Q^+_{k+3})}
	 	\norm{u}^{\alpha_3}_{L_{q}^tL^x_{r}(Q^+_{k+3})} \nonumber \\
	 	&\cdot \left(\varepsilon_0+ \norm{ \nabla u}_{L^t_2L^x_2(Q^+_{k+3})}\norm{u}^{1-\alpha_4}_{L^t_qL^x_r(Q^+_{k+3})}\norm{u}^{\alpha_4}_{L^t_{q_5}L^x_{r_5}(Q^+_{k+3})} \right).
	 \end{align}
 Note we have consistently chosen $(q_n,r_n)$ such that
 $$ \frac{2}{q_n}+\frac{3}{r_n}=\frac{3}{2}, \quad 2\leq q_n\leq \infty, \quad 2\leq r_n\leq 6,\quad \text{  for }n = 1,2,3,4,5.$$
 Thus by Lemma \ref{lem1b}, we know
 $$\norm{u}_{L^t_{q_n}L^x_{r_n}(Q_k^+)}\leq (A_k+E_k)^{1/2} \quad \text{for }n = 1,2,3,4,5.$$

 Substituting \eqref{eqn0b}, \eqref{eqn2},  \eqref{eqn13}, \eqref{eqn5p}, and \eqref{eqn5b} into \eqref{eqn1}, we obtain
	 \begin{align*}
	 	A_k+E_k &\leq  N 2^{bk}\left( \sum_{j=1}^2 \varepsilon_0^{\frac{1}{1-2\tau_j}}(A_{k+3}+E_{k+3})^{\frac{1-3\tau_j}{1-2\tau_j}} +\varepsilon_0^{\frac{2}{3(1-2\tau_1)}}(A_{k+3}+E_{k+3})^{\frac{2(1-3\tau_1)}{3(1-2\tau_1)}}  \right. \nonumber \\
	 	&\quad \left. +  \varepsilon_0^{1+\alpha_3}(A_{k+3}+E_{k+3})^{\frac{1-\alpha_3}{2}} +\varepsilon_0^{1+\alpha_3-\alpha_4}(A_{k+3}+E_{k+3})^{\frac{2-\alpha_3+\alpha_4}{2}} \right).
	 \end{align*}
By Young's inequality, for any $\delta>0$ we have
\begin{align}
	\label{eqn15b}
	A_k+E_k\leq \delta^3(A_{k+3}+E_{k+3}) + \tilde{N}2^{bk}\varepsilon_0  ^{\beta},
\end{align}
where $\tilde{N} = \tilde{N}(\delta,q,r,b)>0$ and $\beta=\beta(q,r)\in (0,1)$.
We multiply both sides of  \eqref{eqn15b} by $\delta^k$ and sum over integer $k$ from $1$ to infinity. By setting $\delta = 3^{-b}$, we make sure the second term on the right-hand side is summable and get
\begin{align*}
	\sum_{k=1}^{\infty}\delta^k(A_k+E_k) &\leq \sum_{k=1}^{\infty}\delta^{k+3}(A_{k+3}+E_{k+3}) +  \tilde{N}\varepsilon_0^{\beta}\sum_{k=1}^{\infty}(2/3)^{bk}\\
	& \leq \sum_{k=3}^{\infty}\delta^{k}(A_{k}+E_k) +  \tilde{N}\varepsilon_0^{\beta}.
\end{align*}
Therefore, we have
\begin{equation*}
	 A_1+E_1\leq \tilde{N}\varepsilon_0^{\beta},
	 \end{equation*}
	 where $\tilde{N} = \tilde{N}(\delta,q,r,b)>0$ and $\beta=\beta(q,r)\in (0,1)$.
Together with $D_{\tilde{q},1}<\varepsilon_0$, we can use Theorem \ref{thm0} to  conclude  that  there exists a universal $\varepsilon_0=\varepsilon_0(q,r)$  sufficiently small such that $u$ is regular in $\overline{Q^+(1/2)}$.
\end{proof}

\noindent {\bf Remark added after the proof:} After we finished this paper, we learned that a result similar to Theorem \ref{thm4} was proved in \cite{181200900W} under a much stronger assumption on the pressure.

\appendix
\section{Interior regularity criterion}
        \label{append}
In the appendix, we show how our proof is adapted to the interior case where $\tilde q$ is allowed to be $1$. We note that Theorem \ref{thm1} was also obtained recently in \cite{170901382H} by using a different proof.

For $\rho>0$, we define the scale invariant quantities $A(\rho)$, $E(\rho)$, $C_{q,r}(\rho)$, and $D_{q,r}(\rho)$ with $B(\rho)$ and $Q(\rho)$ in place of $B^+(\rho)$ and $Q^+(\rho)$.

\begin{theorem}
	\label{thm1} For each pair of exponents  $(q,r)$ satisfying
\begin{equation*}
\frac{2}{q}+\frac{3}{r}<2,\quad 1<q<\infty,\quad \frac{3}{2}<r<\infty, \nonumber
\end{equation*}
there exists a universal constant $\e_0 = \e_0(q,r)$  such that if $(u,p)$ is a pair of suitable weak solution to \eqref{NS1} in $Q$ with $p\in L_{1}(Q)$ and satisfies
\begin{equation}
\label{eqn0}
C_{q,r}(1)+D_{1,1}(1)<\varepsilon_0,
\end{equation}
then $u$ is regular in $\overline{Q(1/2)}$.
\end{theorem}

\begin{lemma}
	\label{lem1}
 For any $\rho>0$ and a pair of exponents $(q,r)$ such that
$$
\frac{2}{q}+\frac{3}{r}=\frac 3 2, \quad 2\leq q\leq \infty, \quad 2\leq r\leq 6,
$$
we have
$$
\rho^{-1/2}\norm{u}_{L_{q}^tL_{r}^x(Q(\rho))} \leq N\left(A(\rho)+E(\rho)\right)^{\frac{1}{2}}.
$$
\end{lemma}
\begin{proof}
 Use the standard interpolation by the Sobolev embedding inequality and H\"{o}lder's inequality.
\end{proof}

\begin{proof}[Proof of Theorem \ref{thm1}]
As before we may assume that $[p]_1(t)=0$ for $t\in (-1,0)$. Also we can find $\hat r\in (3/2,r)$ such that
	\begin{equation}
	\label{eq4a}
	\frac{2}{q}+\frac{3}{\hat{r}} : =2-\tau_1<2,\quad
	\frac{1}q+\frac{2}{\hat{r}} > 1, \quad \frac{1}q+\frac{1}{\hat{r}} > \frac 2 3.
	\end{equation}
As before, we choose $\tau_1$ sufficiently small by reducing $\hat{r}$. Following the beginning part of  proof of Theorem \ref{thm1b}, we have
	\begin{equation}
	\label{eqn1b}
	A_k+E_k \leq N 2^{2k}\int_{Q_{k+1}}\left\{ \abs{u}^2+\abs{u}^3+\abs{pu} \right\}\, dz.
	\end{equation}
	The estimates for the first two terms on the right remain the same.   For the third term, we decompose it in the following way. For each $k$, let $\eta_k(x)$ be a smooth cut-off function supported in $B_{k+2}$, $0 \leq\eta_k\leq 1$ and $\eta_k \equiv 1
$ on $B(\frac{\rho_{k+1}+\rho_{k+2}}{2})$. In the sense of distribution, for a.e. $t\in (-1,0)$, it holds that
$$
\Delta p = D_{ij}(u_iu_j)\quad \text{in}\, B_1.
$$
We consider the decomposition $p = p_k+h_k$,
where $p_k$ is the Newtonian potential of
$D_{ij}(u_iu_j\eta_k).$
Then $h_k$ is harmonic in $B(\frac{\rho_{k+1}+\rho_{k+2}}{2})$.

{\em Estimates for $p_k$:}
Let $q'=\frac{2q}{q+1}\in (1,\infty)$, and $s'\in (1,\infty)$ and $\alpha_1\in (0,1)$ be constants which we will specify later.
By H\"older's inequality,
\begin{equation}
\label{eqn11}
\int_{Q_{k+1}}\abs{up_k}\, dz\leq \norm{u}^{1-\alpha_1}_{L_{q_2}^tL^x_{r_2}(Q_{k+1})}
\norm{u}^{\alpha_1}_{L_{q}^tL^x_{\hat r}(Q_{k+1})}
\norm{p_k}_{L^t_{q'}L^x_{s'}(Q_{k+1})},
\end{equation}
where the exponents satisfy
\begin{equation}
\label{eq4.54a}
\frac{2}{q_2}+\frac{3}{r_2} = \frac{3}{2},
\quad \frac{1}{q'}+\frac{1-\alpha_1}{q_2}+\frac{\alpha_1}{q}=1,
\quad\frac{1}{s'}+\frac{1-\alpha_1}{r_2}+\frac{\alpha_1}{\hat r}=1.
\end{equation}
To make use of Lemma \ref{lem1}, we require $ 2\leq q_2\leq \infty$, which is equivalent to
\begin{equation}
            \label{eqn17}
\frac{1}{q'}+\frac{\alpha_1}{q}\leq 1\leq \frac{1}{q'}+\frac{\alpha_1}{q}+\frac{1-\alpha_1}{2}.
\end{equation}
We will come back to check this condition later.
By \eqref{eq4.54a} we also have
\begin{equation}
\label{eqn7}
\frac{2}{q'}+\frac{3}{s'} = 5-\left((1-\alpha_1)\cdot \frac{3}{2}+\alpha_1(2-\tau_1)\right).
\end{equation}

Using the Calder\'{o}n-Zygmund estimate, we have
\begin{equation}
\label{eqn3}
\norm{p_k}_{L^t_{q'}L^x_{s'}(Q_{k+1})}\leq \norm{p_k}_{L^t_{q'}L^x_{s'}(Q_{k+2})} \leq
N \norm{u}^2_{L^t_{2q'}L^x_{2s'}(Q_{k+2})}.
\end{equation}
By H\"{o}lder's inequality, we have
\begin{equation}
\label{eqn6}
\norm{u}^2_{L^t_{2q'}L^x_{2s'}(Q_{k+2})} \leq \norm{u}_{L_q^tL_{\hat r}^x(Q_{k+2})}\norm{u}_{L_{q_3}^tL_{r_3}^x(Q_{k+2})},
\end{equation}
 where
\begin{equation*}
q_3=\frac {2q}{q-1}\in (2,\infty),\quad r_3=\frac {6q}{q+2}\in (2,6),\quad \frac 2 {q_3}+\frac 3 {r_3}=\frac 3 2,
\end{equation*}
and
\begin{equation*}
\frac 1 {s'}=\frac 1 {\hat r}+\frac 1 {r_3}=\frac 1 {\hat r}+\frac 1 6 +\frac 1 {3 q}.
\end{equation*}
Plugging this into \eqref{eqn7} and using \eqref{eq4a}, we get
$$
\alpha_1= \frac{2\tau_1}{1-2\tau_1}.
$$
Since we have solved for $\alpha_1$, we can now go back to verify \eqref{eqn17},
which is equivalent to
$$
\frac{1}{2}+\frac{1}{2q}+\frac{\alpha_1}{q}\leq 1\leq 1+\frac{1}{2q}+\frac{\alpha_1}{q}-\frac{\alpha_1}{2}.
$$
This indeed is satisfied when $\tau_1$ is sufficiently small.
Thus by  Lemma \ref{lem1}, \eqref{eqn11}, \eqref{eqn3}, and \eqref{eqn6}, we have
\begin{align}
\label{eqn5}
\int_{Q_{k+1}}\abs{up_k}\, dz  \leq & N (A_{k+2}+E_{k+2})^{\frac{1-3\tau_1}{1-2\tau_1}}
\norm{u}^{\frac 1 {1-2\tau_1}}_{L_q^tL_r^x(Q_{k+2})}.
\end{align}

{\em Estimates for $h_k$:}
By H\"{o}lder's inequality, we have
\begin{equation}
\label{eqn11g}
\int_{Q_{k+1}}\abs{uh_k}\, dz\leq \norm{u}_{L_{\infty}^tL^x_{2}(Q_{k+1})}
\norm{h_k}_{L^t_{1}L^x_{2}(Q_{k+1})}.
\end{equation}

Recall that  $h_k$ is harmonic in $B(\frac{\rho_{k+1}+\rho_{k+2}}{2})$.
By the fact that any Sobolev norm of a harmonic function in a smaller ball can be estimated by any of its $L_p$ norm in a greater ball, we know
\begin{equation*}
\norm{h_k}_{L_{2}^x(B_{k+1})}\leq N 2^{b k}\norm{h_k}_{L_1^x(B_{k+2})},\nonumber
\end{equation*}
where $b>0$ is a constant.
Integrating in $t\in (-\rho^2_{k+1},0)$ we have
\begin{equation}
\label{eqn4q}
\norm{h_k}_{L_1^tL_{2}^x(Q_{k+1})}\leq N 2^{b k}\norm{h_k}_{L_1^tL_1^x(Q_{k+2})}\leq N 2^{b k}(\norm{p_k}_{L_1^tL_1^x(Q_{k+2})}+\norm{p}_{L_1^tL_1^x(Q_{k+2})}),
\end{equation}
where the second term is small by condition \eqref{eqn0}. By the Calder\'{o}n-Zygmund estimate and H\"{o}lder's inequality, for any $\tilde{r}>1$ we have
\begin{equation}
\label{eqn5g}
\norm{p_k}_{L_1^tL_{1}^x(Q_{k+2})}\leq
N\norm{p_k}_{L_1^tL_{\tilde{r}}^x(Q_{k+2})}\leq
N\norm{u}^2_{L_1^tL_{2\tilde{r}}^x(Q_{k+2})}.
\end{equation}
For $q>1$, we claim the following interpolation holds for some $\alpha,q_4,r_4>0$:
\begin{equation}
	\label{eqn2p}
	\norm{u}^2_{L_1^tL_{2\tilde{r}}^x(Q_{k+2})} \leq N\norm{u}^{2-\alpha}_{L_q^tL_{r}^x(Q_{k+2})}\norm{u}^{\alpha}_{L_{q_4}^tL_{r_4}^x(Q_{k+2})},
\end{equation}
where
$$
q_4\in [2,\infty],\quad r_4\in [2,6],\quad \frac{2}{q_4}+\frac{3}{r_4}=\frac{3}{2},
$$
and they need to satisfy
$$
\frac{2-\alpha}{q}+\frac{\alpha}{q_4}\le 1,\quad
\frac{2-\alpha}{r}+\frac{\alpha}{r_4}\le 1.
$$
Indeed, we can choose $(\alpha,q_4,r_4)$ in the following way:
when $1<q\leq 2$, we set $q_4 = \infty$, $r_4 =2$, and $\alpha = 2-q$;
when $q>2$, we set $q_4 =2$, $r_4 = 6$, and $\alpha=6/7$. Note that in both cases we have $\alpha<1$.

Now we plug \eqref{eqn4q}, \eqref{eqn5g}, and \eqref{eqn2p} into \eqref{eqn11g} to obtain
\begin{align}
	\label{eqn5k}
	\int_{Q_{k+1}}\abs{uh_k}\, dz  \leq & N2^{bk}\left((A_{k+2}+E_{k+2})^{\frac{1}{2}} \norm{p}_{L_1^tL_1^x(Q_{k+2})}\right. \nonumber \\
	&+  \left. (A_{k+2}+E_{k+2})^{\frac{1+\alpha}{2}}
	\norm{u}^{2-\alpha}_{L_q^tL_r^x(Q_{k+2})}\right).
\end{align}
By \eqref{eqn1b}, \eqref{eqn2}, \eqref{eqn13}, \eqref{eqn5}, \eqref{eqn5k}, and condition \eqref{eqn0} we conclude that
\begin{align*}
A_k+E_k &\leq  N 2^{bk}\left( \varepsilon_0^{\frac{1}{1-2\tau_1}}(A_{k+2}+E_{k+2})^{\frac{1-3\tau_1}{1-2\tau_1}}+ \varepsilon_0^{\frac{2}{3(1-2\tau_1)}}(A_{k+2}+E_{k+2})^{\frac{2(1-3\tau_1)}{3(1-2\tau_1)}}
 \right. \nonumber \\
&\quad \left. +\varepsilon_0(A_{k+2}+E_{k+2})^{\frac{1}{2}} +\varepsilon_0^{2-\alpha}(A_{k+2}+E_{k+2})^{\frac{1+\alpha}{2}}
 \right).
\end{align*}
Now similar to the proof in Section \ref{sec4}, we obtain
$$
A_1+E_1\leq \tilde{N}\varepsilon_0^{\beta}.
$$
Together with $D_{1,1}<\varepsilon_0$, we can apply \cite[Theorem 1.5]{Phuc} to  conclude that there exists a universal $\varepsilon_0=\varepsilon_0(q,r)$  sufficiently small such that $u$ is regular in $\overline{Q(1/2)}$.
\end{proof}

\bibliographystyle{amsplain}
\providecommand{\bysame}{\leavevmode\hbox to3em{\hrulefill}\thinspace}
\providecommand{\MR}{\relax\ifhmode\unskip\space\fi MR }
\providecommand{\MRhref}[2]{%
  \href{http://www.ams.org/mathscinet-getitem?mr=#1}{#2}
}
\providecommand{\href}[2]{#2}

\bibliographystyle{amsplain}
\providecommand{\bysame}{\leavevmode\hbox to3em{\hrulefill}\thinspace}
\providecommand{\MR}{\relax\ifhmode\unskip\space\fi MR }
\providecommand{\MRhref}[2]{%
  \href{http://www.ams.org/mathscinet-getitem?mr=#1}{#2}
}
\providecommand{\href}[2]{#2}

\end{document}